\documentclass[10pt, oneside]{amsart}
 \usepackage[text={5.58in,8.5in},centering,letterpaper,dvips]{geometry}
\usepackage{graphicx}
\usepackage[usenames,dvipsnames]{xcolor}
\usepackage{amsfonts}
\usepackage{epsf}
\usepackage{amssymb}
\usepackage{amsmath}
\usepackage{amscd}
\usepackage{pdfpages}
\usepackage{fancyhdr}
\usepackage{setspace}
\usepackage{hyperref}
\usepackage[all]{xy}
\usepackage{verbatim}

\usepackage{xparse}
\usepackage{tikz}
\usetikzlibrary{matrix,backgrounds,arrows,matrix,positioning}
\pgfdeclarelayer{myback}
\pgfsetlayers{myback,background,main}

\tikzset{mycolor/.style = {line width=1bp,color=#1}}%
\tikzset{myfillcolor/.style = {draw,fill=#1}}%

\NewDocumentCommand{\highlight}{O{blue!40} m m}{%
\draw[mycolor=#1] (#2.north west)rectangle (#3.south east);
}

\NewDocumentCommand{\fhighlight}{O{blue!40} m m}{%
\draw[myfillcolor=#1] (#2.north west)rectangle (#3.south east);
}

\newtheorem{theorem}{Theorem}[section]
\newtheorem{proposition}[theorem]{Proposition}
\newtheorem{lemma}[theorem]{Lemma}

\newtheorem{question}[theorem]{Question}
\newtheorem{corollary}[theorem]{Corollary}

\newtheorem{problem}[theorem]{Problem}

\newtheorem{example}[theorem]{Example}

\newcommand{\calc}{\mathcal{C}}
\newcommand{\calg}{\mathcal{G}}

\newcommand{\Z}{\mathbb{Z}}

\newcommand{\Q}{\mathbb{Q}}

\newcommand{\bi}{\begin{itemize}}
\newcommand{\ei}{\end{itemize}}
\newcommand{\be}{\begin{enumerate}}
\newcommand{\ee}{\end{enumerate}}

\newcommand{\pd}{\partial}

\newcommand{\K}{\mathcal{K}}

\newcommand{\ups}{\upsilon}
\newcommand{\co}{\colon}

\makeatletter
\newtheorem*{rep@theorem}{\rep@title}
\newcommand{\newreptheorem}[2]{%
\newenvironment{rep#1}[1]{%
 \def\rep@title{#2 \ref{##1}}%
 \begin{rep@theorem}}%
 {\end{rep@theorem}}}
\makeatother

\newreptheorem{theorem}{Theorem}
\newreptheorem{lemma}{Lemma}
\newreptheorem{question}{Question}
\newreptheorem{corollary}{Corollary}

\begin{document}
\rhead{\thepage}
\lhead{\author}
\thispagestyle{empty}
\raggedbottom
\pagenumbering{arabic}
\setcounter{section}{0}



\title{Doubly slice knots with low crossing number}
\author{Charles Livingston}\author{Jeffrey Meier}
\thanks{The first author was supported by a grant from the Simons Foundation.  The second author was supported by the National Science Foundation under grant DMS-1400543. \\ 
}

\address{Charles Livingston: Department of Mathematics, Indiana University, Bloomington, IN 47405 }
\email{livingst@indiana.edu}
\address{Jeffrey Meier:  Department of Mathematics, Indiana University, Bloomington, IN 47405 }

\email{ jlmeier@indiana.edu}

\begin{abstract} 
 A knot in $S^3$ is doubly slice if it is the cross-section of an unknotted two-sphere in $S^4$.  For low-crossing knots, the most complete work to date gives a classification of doubly slice knots through 9 crossings.   We extend that work through 12  crossings, resolving all but four cases among the 2,977 prime knots in that range.  The techniques involved in this analysis include considerations of the Alexander module and signature functions as well as calculations of the twisted Alexander polynomials for higher-order branched covers.  We give explicit illustrations of the double slicing for each of the 20 knots shown to be smoothly doubly slice. We place the study of doubly slice knots in a larger context by introducing the \emph{double slice genus} of a knot.
\end{abstract}

\maketitle

\section{Introduction}\label{section:intro}

In 1962,  Fox included the following question in his list of  problems in knot theory~\cite{fox:problems}.

\begin{question}
Which slice knots and weakly slice links can appear as the cross-sections of the unknotted $S^2$ in $S^4$.
\end{question}

\noindent Such a knot is called \emph{doubly slice}.  Many of the techniques that have been successful in the study of slice knots and knot concordance over the last 50 years have applications to the study of doubly slice knots and double null concordance of knots.  Nevertheless, doubly slice knots remain far less  understood than their slice counterparts.

The goal of this note is to address Fox's question for prime knots with 12 or fewer crossings.  A precedent for this work was set in 1971 when Sumners showed that for prime knots with nine or fewer crossings, there is only one prime doubly slice knot, namely, the  knot $9_{46}$~\cite{sumners:invertible}.

There are  158  known  prime slice   knots with 12 or fewer crossings, and it is unknown whether the  knot $11_{n34}$ is slice.  Of these 159 knots, we show that at least 20, but no more than 24, are smoothly doubly slice.  

\begin{theorem}\label{thm:Smooth}
The following knots are smoothly doubly slice.
$$
\begin{array}{lllllll}
	9_{46} & 10_{99} & 10_{123} & 10_{155} & 11_{n42} & 11_{n49} & 11_{n74} \\
	12_{a0427} & 12_{a1105}  & 12_{n0268} & 12_{n0309} & 12_{n0313} & 12_{n0397} & 12_{n0414} \\
	 12_{n0430} & 12_{n0605} & 12_{n0636} & 12_{n0706} & 12_{n0817} & 12_{n0838} 
\end{array}
$$
Furthermore, the following are the only other prime knots with 12 or fewer crossings   that could possibly be smoothly doubly slice.
$$
\begin{array}{lllllll}
	11_{n34} & 11_{n73} & 12_{a1019} & 12_{a1202} 
\end{array}
$$
\end{theorem}

Our contributions to this computation include the following.
\begin{itemize}
\item The first application of twisted Alexander polynomials to obstruct double sliceness.

\item The first low-crossing examples of slice knots with non-vanishing signature function.  

\item Explicit constructions of unknotted embeddings of $S^2$ into $S^4$ with equatorial cross-section isotopic to each of the 20 knots on the list.
\end{itemize}

As mentioned above, $9_{46}$ was first double sliced by Sumners~\cite{sumners:invertible}, while $11_{n42}$ was shown to be doubly slice in~\cite{carter-kamada-saito}. The double slicing of $10_{123}$ included below was shown to the second author by Donald, who has contributed to the study of double slice knots by studying the problem of embedding 3--manifolds into $S^4$~\cite{donald:embedding}.  
 
We show below that the Conway knot $11_{n34}$ is topologically doubly slice (see Section~\ref{section:topological}), but it is unknown whether it can be smoothly sliced or double sliced.

\subsection{A brief history of doubly slice knots}\ 

The study of slice knots is naturally placed in the context of the concordance group $\calc$ and the homomorphism $\phi\co \calc \to \calg$, where $\calg$ is the algebraic concordance group,  defined  and classified  by Levine~\cite{levine:invariants, levine:groups}.  There are analogous groups $\calc_{ds}$ and $\calg_{ds}$ defined in the context of doubly slice knots; however, Levine's classification of $\calg$ does not carry over to $\calg_{ds}$, and there are other complications that make $\calc_{ds}$ and $\calg_{ds}$ difficult to study.

It is known that the kernel of the canonical map $\calg_{ds} \to \calg$ in infinitely generated~\cite{cha-liv:signature}, but beyond that, the structure of $\calg_{ds}$ remains a mystery.  (See, however,~\cite{stoltzfus-bayer-fluckiger, stoltzfus:unraveling,stoltzfus:double}.)  Furthermore, it can be shown using Casson-Gordon invariants that there are algebraically doubly slice knots that are not topologically doubly slice~\cite{gilmer-livingston:embedding}.  Friedl developed further metabelian invariants that can be used to obstruct double sliceness~\cite{friedl:eta}.

As in the study of slice knots, there is an important distinction  between the smooth and topologically locally flat categories.  However, this distinction does not feature prominently in our work here; we find no low-crossing examples of knots that are topologically doubly slice but not smoothly doubly slice, even though  such knots have been shown to exist~\cite{meier:double}. Other interesting constructions in the study of doubly slice knots include the fibered examples of Aitchison-Silver~\cite{ait-silver} and the extension of the Cochran-Teichner-Orr filtration to topologically doubly slice knots by Kim~\cite{kim:new}.

\subsection{Organization}\ 

In Sections~\ref{section:algebraic} and~\ref{section:topological}, we discuss obstructions to double slicing knots coming from the algebraic and topological categories, respectively.  In Section~\ref{section:slicing}, we discuss some techniques that can be used to construct double slicings of knots in either the topological or smooth categories.  In Section~\ref{sec:genus}, we place the study of doubly slice knots in context by considering knots as cross-sections of unknotted surfaces in $S^4$.  

\section{Algebraic obstructions to double slicing knots}\label{section:algebraic}

In this section, we will present three algebraic obstructions to double slicing a knot. These are  applied to obtain an initial list of prime knots with at most twelve crossings that could potentially be 
doubly slice.

\subsection{Hyperbolic torsion coefficients}\ 

A knot $K$ in $S^3$ is said to be \emph{algebraically doubly slice} if there exists a Seifert matrix $A_K$ for $K$ that has the form
$$A_K = \begin{bmatrix}
 0 & B_1 \\
 B_2 & 0 
 \end{bmatrix},$$
where $B_1$ and $B_2$ are square matrices of equal dimension.  Matrices of this form are called \emph{hyperbolic} and have been studied by Levine~\cite{levine:hyperbolic} and others~\cite{cha-liv:signature,stoltzfus:double}.  If $K$ is (smoothly or topologically) doubly slice, then $K$ is algebraically double slice~\cite{sumners:invertible}.  

Let $A_K$ be a hyperbolic Seifert matrix for $K$.  Then, 
$$A_K+A_K^T = \begin{bmatrix}
 0 & B \\
 B^T & 0 
 \end{bmatrix},$$
where $B=B_1+B_2^T$.  The matrix $B \oplus B$ is a presentation matrix for $H_1(\Sigma_2(K))$.  It follows that $H_1(\Sigma_2(K))$ splits as a direct sum $G \oplus G$, where  $G$ is presented by the matrix $B$. Thus, we have our first obstruction.

\begin{proposition}\label{prop:hom}
Let $K$ be a knot in $S^3$.  If $K$ is algebraically doubly slice, then, for some finite group $G$, $H_1(\Sigma_2(K))=G\oplus G$.
\end{proposition}

Of the 2,977 prime knots with at most 12 crossings, 62 knots  satisfy Proposition~\ref{prop:hom}. Furthermore, if $K$ is algebraically doubly slice, then $K$ is algebraically slice. Among these 62 knots, there are 36 that are algebraically slice.  These knots form our short-list of candidates to be algebraically doubly slice  and are shown below.

$$
\begin{array}{lllllllll}
	9_{41} & 9_{46} & 10_{99} & 10_{123} & 10_{153} & 10_{155} & 11_{n34} & 11_{n42} \\
	11_{n49} & 11_{n73} & 11_{n74} & 11_{n116} & 12_{a0427} & 12_{a1019} & 12_{a1105} & 12_{a1202} \\
	12_{n0019}  & 12_{n0210} & 12_{n0214} & 12_{n0257} & 12_{n0268} & 12_{n0309} & 12_{n0313} & 12_{n0318} \\
	12_{n0397} & 12_{n0414} & 12_{n0430} & 12_{n0440} & 12_{n0582} & 12_{n0605} & 12_{n0636} & 12_{n0706} \\
	12_{n0813} & 12_{n0817} & 12_{n0838} & 12_{n0876}
\end{array}
$$

\subsection{The signature function}\label{subsec:signature}\ 

Let $K$ be a knot in $S^3$  with Seifert matrix  $A_K$. Let $\omega$ be a unit complex number, and consider the matrix
$$(1-\omega)A_K+(1-\overline\omega)A_K^T.$$
Denote by  $\sigma_\omega(K)$   the signature of this matrix.  Note that this matrix will be non-singular provided that $\Delta_K(\omega)\not=0$, where $\Delta_K(t)$ is the Alexander polynomial of $K$.  In any event, $\sigma_K(\omega)$ is a well-defined knot invariant for any unit complex number $\omega$.  See~\cite{gordon:knot_theory} for details.  It is well-known that $|\sigma_K(\omega)|\leq 2g_4(K)$ whenever $\Delta_K(\sigma)\not=0$.  Thus, if $K$ is algebraically slice, then $\sigma_\omega(K)=0$ away from the roots of the Alexander polynomial.  Moreover, we have the following.

\begin{proposition}\label{prop:sign}
Let $K$ be a knot in $S^3$.  If $K$ is algebraically doubly slice, then $\sigma_\omega(K)=0$ for any unit complex number $\omega$.
\end{proposition}

In fact, we can consider these signature invariants as a function $\sigma(K):S^1\to\Z$, defined by $\sigma(K)(\omega)=\sigma_\omega(K)$, called the \emph{signature function}.  If a knot $K$ satisfies Proposition~\ref{prop:sign}, we say that the signature function for $K$ \emph{vanishes}.

\begin{example}
	Let $K=12_{n0582}$.  Then, $\Delta_K(t) = (t^2-t+1)^2$, and the roots of $\Delta_K(t)$ are contained on the unit circle.  Since $K$ is slice, we know that $\sigma_K(\omega)=0$ away from these roots.  However, if we consider the roots, $\zeta$ and $\overline\zeta$, where $\zeta$ is a sixth root of unity, we can compute that $\sigma_{\zeta}(K)=\sigma_{\overline\zeta}(K)=-1$. (Note that this calculation depends on a Seifert matrix $A_K$, but any choice will do and we do not include the details here.) 
It follows from Proposition~\ref{prop:sign} that $K$ cannot be algebraically doubly slice.
\end{example}

\begin{example}
	Let $K=12_{n0813}$.  Then, $\Delta_K(t) = (t-2)(2t-1)(t^2-t+1)^2$. Two of the roots of $\Delta_K(t)$ are primitive sixth roots of unity; the other two roots do not lie on the unit circle, so no information can be gained by considering them.  If we consider the roots of unity, we find that $\sigma_{\zeta}(K)=\sigma_{\overline\zeta}(K)=+1$. (Again, we have used some matrix $A_K$ for this calculation.)  It follows from Proposition~\ref{prop:sign} that $K$ cannot be algebraically doubly slice.
\end{example}

Thus, we remove $12_{n0582}$ and $12_{n0813}$ from our list of potentially algebraically doubly slice knots.

\subsection{The Alexander module}\label{subsec:module}\ 

Continuing, let $K$ be a knot in $S^3$  and let $X_\infty(K)$ denote the infinite cyclic cover of $S^3\setminus K$.  The group $H_1(X_\infty(K))$ can be regarded as a $\Lambda$--module, where $\Lambda = \Z[t,t^{-1}]$.  This $\Lambda$--module is called the \emph{Alexander module} and is presented by the matrix $V_K=A_K-tA_K^T$.  

Sumners   obstructed  $9_{41}$ from being doubly slice by carefully analyzing the module structure of $H_1(X_\infty(K))$.   We follow a similar approach to  analyze two more knots.
 
We begin by switching to coefficients in the  finite field with $p$ elements,  ${\mathbb{Z}_p}$.  In this case,  $ H_1(X_\infty(K),{\mathbb{Z}_p} )$ is a module over a PID, $\Lambda_p = {\mathbb{Z}_p}[t,t^{-1}]$.   We now have that  if $K$ is doubly slice, then as a $\Lambda_p$--module,  $$H_1(X_\infty(K),{\mathbb{Z}_p} )\cong \bigoplus_i \left( \Lambda_p / \left<f_i(t) \right> \oplus \Lambda_p /  \left<f_i(t^{-1}) \right>\right) $$ for some set of polynomials $f_i(t) \in \Lambda_p$.

\begin{example}
Let $K=11_{n116}$, which has $\Delta_K(t) = (1+t-t^2)(-1+t+t^2)$. Using the Seifert form $V_K$ taken from KnotInfo~\cite{cha-liv:knotinfo} and working with  ${\mathbb{Z}_2}	$--coefficients, we find that as a $\Lambda_2$--module,  $$  H_1(X_\infty(K),{\mathbb{Z}_2} )\cong \Lambda_2  /  \left<( 1+t+t^2)^2  \right> .$$ This does not decompose as a nontrivial  direct sum of modules, so $K=11_{n116}$ cannot be doubly slice.
\end{example}

\begin{example}
Let $K=12_{n0876}$, which has $\Delta_K(t) = (-2+4t-2t^2+t^3)(-1+2t-4t^2+2t^3)$.  Again using the Seifert form $V_K$ taken from KnotInfo, but now working with  ${\mathbb{Z}_3}$--coefficients, we compute that as a $\Lambda_3$--module,  
	$$  H_1(X_\infty(K),{\mathbb{Z}_3} )\cong   \Lambda_3/  \left<(  1+t )^2  \right> \oplus \Lambda_3/  \left<(  1 + t^2 )^2  \right> .$$ This does not decompose further, so $12_{n0876}$ cannot be doubly slice. 
\end{example}

\subsection{Algebraic conclusions}\ 

In conclusion, consideration of the torsion invariants reduced our search for doubly slice knots  to a set of $36$ knots.  An analysis of the signature function removed another two, and an examination of Alexander modules eliminate three more, including the one found by Sumners.  Of the remaining 31 knots, we will use the techniques described in Section~\ref{section:slicing} to show that one is topologically doubly slice and 20 are smoothly doubly slice.  It follows that these 21 knots are algebraically doubly slice, leaving us with only 10 knots that may or may not be algebraically doubly slice.

\begin{question}
Are any of the following knots algebraically doubly slice?

$$
\begin{array}{lllllllll}
	 10_{153} & 11_{n73} &  12_{a1019} & 12_{a1202} & 12_{n0019}    \\
	 12_{n0210} & 12_{n0214} & 12_{n0257} & 12_{n0318} & 12_{n0440}   \\
	 
\end{array}
$$\vskip.2in
	
\end{question}

\section{Topological obstructions to double slicing knots}\label{section:topological}

We now move from abelian to metabelian invariants.  We begin by quickly recalling the twisted polynomial.   Let $M_q(K)$ be the $q$--fold cyclic cover of $S^3\setminus K$, let $\Sigma_q(K)$ be the branched cyclic cover,  and let $\rho \co H_1(\Sigma_q(K)) \to \Z_p$ be a homomorphism, where $q$ is a prime power and $p$ is an odd prime. Let $\Gamma_p=\Q(\zeta_p)[t,t^{-1}]$, where $\zeta_p$ is a primitive $p^\text{th}$--root of unity.  As described in~\cite{kirk-liv:twisted}, there is an associated {\it twisted Alexander polynomial} $\Delta_{K,\rho}(t) \in \Gamma_p$.  This polynomial is well-defined up to multiplication by a unit in $\Gamma_p$.  Given $f(t)\in\Gamma_p$, let $\overline{f(t)}$ denote the result of complex conjugation of the coefficients of $f(t)$.


A result of~\cite{kirk-liv:twisted} states the following.

\begin{theorem}\label{thm:kl}If $K$ is slice, then there is a subgroup $H \subset  H_1(\Sigma_q(K))$ satisfying the following properties.
\begin{enumerate}
\item  $|H|^2  =  |H_1(\Sigma_q(K))|.$\vskip.05in

\item The subgroup $H$ is invariant under the action of the deck transformation of $\Sigma_q(K)$.
\vskip.05in
\item For all $\rho \co H_1(\Sigma_q(K)) \to \Z_p$ satisfying $\rho(H) = 0$, one has  $\Delta_{K,\rho}(t) = f(t)\overline{f(t^{-1})}$.
\end{enumerate}

\end{theorem}

If $K$ is doubly slice, then it satisfies strengthened conditions.

\begin{theorem}\label{thm:kl-double}
If $K$ is doubly slice, then there is a splitting $   H_1(\Sigma_q(K)) \cong H_1 \oplus H_2$ satisfying the following properties.
\begin{enumerate}
\item  $H_1 \cong H_2$.\vskip.05in
\item The subgroups $H_1$ and $H_2$ are invariant under the action of the deck transformation of $\Sigma_q(K)$.
\vskip.05in
\item For all $\rho \co H_1(\Sigma_q(K)) \to \Z_p$ for which $\rho(H_1) = 0$ or $\rho(H_2) = 0$, one has  that $\Delta_{K,\rho}(t) = f(t)\overline{f(t^{-1})}$.
\end{enumerate}
\end{theorem}

\begin{proof}The proof is very similar to that of Theorem~\ref{thm:kl} in~\cite{kirk-liv:twisted}, so we just summarize it here.

In Theorem~\ref{thm:kl}, the subgroup $H$ can be taken as the kernel of the inclusion $\Sigma_q(K) \to \overline{W}_q(D)$, where $\overline{W}_q(D)$ is the $q$--fold branched cover of $B^4$ over a slice disk $D$ of $K$.  In the case that $K$ is doubly slice, the $q$--fold branched cover $\Sigma_q(K)$ embeds in $S^4$, since $S^4$ is the  the $q$--fold branched cover of $S^4$ over the (unknotted) double slicing 2--sphere for $K$.  
  It follows that $\Sigma_q(K)$ splits $S^4$ into manifolds $Y_1$ and $Y_2$. 
  
   The subgroups $H_1$ and $H_2$ can be taken as the kernels of the inclusions $H_1(\Sigma_q(K)) \to Y_1$ and $H_1(\Sigma_q(K)) \to Y_2$.  The direct sum decomposition arises from the Meyer-Vietoris Theorem;  the fact that $H_1 \cong H_2$ follows from duality, as first noticed by Hantzche~\cite{hantzche}.

The rest of the argument follows identically to that in~\cite{kirk-liv:twisted}.
\end{proof}

Equipped with Theorem~\ref{thm:kl-double}, we are ready to prove our second result.

\begin{theorem}\label{thm:NotTop}
	The following knots are not topologically doubly slice, but might be algebraically doubly slice.

$$
\begin{array}{lllllll}
	 10_{153} & 12_{n0019}  &12_{n0210}& 12_{n0214} \\
	 12_{n0257} & 12_{n0318} & 12_{n0440} 
\end{array}
$$

\end{theorem}

\begin{proof}

The proof is nearly identical in each case, so we describe only one case in detail.

Let $K=10_{153}$.  Then $H_1(\Sigma_3(K))\cong(\Z_7)^2$, and the action of the deck transformation splits the homology as $E_2\oplus E_4$. Here $E_2$ is the $2$--eigenspace of the action of the deck transformation on $  H_1(\Sigma_3(K))$ and $E_4$ is the $4$--eigenspace.  Notice that $2^3 = 4^3 = 1 \mod 7$.

  Let $\rho_2:(\Z_7)^2\to E_2$ denote projection onto $E_2$, so $\rho_2\vert_{E_4}\equiv 0$, and let $\Delta_{K,\rho_2}(t)$ denote the associated twisted Alexander polynomial.  Then, we have
$$ \Delta_{K,\rho_2}(t)=(-t^2+\omega t+1)(-t^2+\omega t+1),$$
where $\omega=\zeta^4+\zeta^2+\zeta+1$ for a $7^\text{th}$--root of unity $\zeta$.  One easily checks that $\omega=\overline\omega$, so $\Delta_{K,\rho_2}(t) = f(t)\overline{f(t)}$.

On the other hand, if one considers the other projection $\rho_4:H_1(\Sigma_3(K))\to E_4$, so that $\rho_4\vert_{E_2}\equiv 0$, one finds that the associated twisted polynomial is given by
$$\Delta_{K,\rho_4}(t)=t^4+3t^2+1.$$

The following lemma states that  $t^4+3t^2+1$ is irreducible in $\Gamma_7$.  It follows from  Theorem~\ref{thm:kl-double}  that $K$ cannot be topologically doubly slice,  since the twisted polynomials associated to this metabolizing representation do  not factor as   norms. 

\begin{lemma}
The polynomial $p(t)=t^4+3t^2+1$ is irreducible in $\Gamma_7$.
\end{lemma}

\begin{proof} If $\alpha \in  \mathbb{Q}(\zeta_7)$ is a root of $p(t)$, then so is $\alpha^{-1}$.  Thus, if $p(t)$ has a linear factor, it has two distinct linear factors, and hence it has a quadratic factor.  So, suppose that $p(t)$ factors into two quadratic polynomials.  One can assume the factorization is of the form $$p(t) = (t^2 + at + b)(t^2 + a't +b').$$  By examining coefficients, the factorization further simplifies to be of the form
$$p(t) = (t^2 + at + b)(t^2- at +b),$$  where   
$b= \pm 1$ and $a^2 = 2b-3$. 
If $b = 1$, then $a^2 = -1$.    If $b = -1$, then $a^2 = -5$.  Thus, the proof is  completed by showing that ${\mathbb{Q}(\zeta_7)}$ contains neither $\sqrt{-1} $ nor $\sqrt{-5}$.

The Galois group of ${\mathbb{Q}(\zeta_p)}$ is cyclic, isomorphic to $ {\mathbb{Z}}_{p-1}$, and thus contains a unique index two subgroup.  If follows that  ${\mathbb{Q}(\zeta_p)}$ contains a unique quadratic extension of $ {\mathbb{Q}}$.  A standard result in number theory (see~\cite{marcus}) states that this field is ${\mathbb{Q}}(\sqrt{  p})$ or ${\mathbb{Q}}(\sqrt{  -p})$, depending on whether $p$ is congruent to 1 or 3 modulo 4, respectively.  This quickly yields the desired contradiction; for instance, it is clear that  $  {\mathbb{Q}}(\sqrt{  -5}) \not \subseteq {\mathbb{Q}}(\sqrt{  -7})$.
\end{proof}

It follows from  Theorem~\ref{thm:kl-double}  that $K$ cannot be topologically doubly slice,  since the twisted polynomials associated to this metabolizing representation does not factor as a norm.

The general  proof of  Theorem~\ref{thm:NotTop} proceeds by checking that each of the relevant twisted Alexander polynomials does not factor as a norm.   The pertinent information needed to verify the result for the other knots is described in Table~\ref{table:Twisted}.  The Maple program developed in conjunction with~\cite{herald-kirk-liv:twisted} was used to find the twisted polynomials and Maple could also be used to check  the factoring conditions.  

The knot $12_{n0210}$ was shown to not to be topologically slice in~\cite{herald-kirk-liv:twisted} using twisted polynomials, and hence it is not topologically doubly slice.

\end{proof}

\begin{table}[h!]\label{table:Twisted}
\renewcommand{\arraystretch}{2.5}
\tiny
\begin{tabular}{|c|c|c|}
\hline
\textbf{Knot} & \textbf{Cover Homology} &  \textbf{Irreducible Twisted Polynomial} \\
\hline
$10_{153}$ & $\Sigma_3(K)\cong(\Z_7)^2$ & $\Delta_{K,\rho_4}(t)=t^4+3t^2+1$  \\
\hline
$12_{n0019}$ & $\Sigma_3(K)\cong(\Z_{13})^2$ & $\Delta_{K,\rho_9}(t)  =  t^4 
  +2t^2+1 $ \\
 & $\zeta^{13}=1$ & $+t^3\left(\zeta^{11}+\zeta^9+\zeta^8+\zeta^7+2\zeta^6+2\zeta^5+\zeta^3+2\zeta^2+\zeta+1\right)$ \\
  && $+t\left(\zeta^{11}-\zeta^9+\zeta^8+\zeta^7-\zeta^3-\zeta\right)$ \\
\hline
$12_{n0214}$ & $\Sigma_3(K)\cong(\Z_7)^2$ & $\Delta_{K,\rho_2}(t)= -29 t^4 +\left(31+ 8 \zeta + 8 \zeta^2 + 8 \zeta^4\right) $ \\ 
& $\zeta^7=1$ & $+ t^3 \left(-27 + 37 \zeta + 37 \zeta^2 + 37 \zeta^4\right)$ \\
&& $ + t \left(48 + 47 \zeta + 47 \zeta^2 + 47 \zeta^4\right) $ \\
&& $+ t^2 \left(17 + 68 \zeta + 68 \zeta^2 + 68 \zeta^4\right)$ \\
\hline
$12_{n0257}$ & $\Sigma_3(K)\cong(\Z_{13})^2$ & $\Delta_{K,\rho_9}(t)=- 13 t^4 +13 $\\
& $\zeta^{13}=1$ & $+ t^3 \left(37 + 48 \zeta + 21 \zeta^2 + 48 \zeta^3 + 21 \zeta^5 + 21 \zeta^6 + 14 \zeta^7 + 
    14 \zeta^8 + 48 \zeta^9 + 14 \zeta^{11}\right)$
    \\
&& $+ t^2 \left(39 + 78 \zeta + 13 \zeta^2 + 78 \zeta^3 + 13 \zeta^5 + 13 \zeta^6 + 65 \zeta^7 + 
    65 \zeta^8 + 78 \zeta^9 + 65 \zeta^{11}\right)$
     \\
&& $ +t \left(11 + 48 \zeta + 34 \zeta^2 + 48 \zeta^3 + 34 \zeta^5 + 34 \zeta^6 + 27 \zeta^7 + 27 \zeta^8 + 
    48 \zeta^9 + 27 \zeta^{11}\right) $ \\
\hline
$12_{n0318}$ & $\Sigma_3(K)\cong(\Z_7)^2$ & $\Delta_{K,\rho_2}(t) = 1 + 3 t^2 + t^4$ \\
& $\zeta^7=1$ & $ + t \left(3 - \zeta - \zeta^2 - \zeta^4\right) $ \\
&& $+ t^3 \left(4 + \zeta + \zeta^2 + \zeta^4\right)$ \\
\hline
$12_{n0440}$ & $\Sigma_3(K)$ & $\Delta_{K,\rho_2}(t) = t^4-3t^3+6t^2-3t+1$ \\
& $\cong(\Z_2)^4\oplus(\Z_7)^2$ &  \\
\hline
\end{tabular}\vskip.1in
\caption{Twisted Alexander polynomial calculations.}
\end{table}

\section{Double slicing knots}\label{section:slicing}

In this section, we discuss some techniques that can be used to show that a knot is doubly slice. We will address the issue of double sliceness in both the smooth and the locally flat settings.  

\subsection{Band systems}\label{subsec:bands}\ 

In~\cite{donald:embedding}, Donald showed that if a knot can be sliced by two different sequences of band moves, and if the bands are related in a certain way, then combining the two ribbon disks yields an unknotted 2--sphere.  In this section we present a concise treatment of this result.

Let $L$ be a link in $S^3$ and let $b$ be the image of a 2--disk embedded in $S^3$ such that $L  \cap b$ consists of two disjoint arcs in $\partial b$.  We refer to such a $b$ as a {\it band}  and denote by $L*b$ the link formed as the closure of   $(L \cup  \partial b) \setminus (L\cap b)$.  Notice that $(L*b)*b = L$;  also, if $b$ and $c$ are disjoint, then $(L*b)*c = (L*c)*b$, so we can write both as $L*b*c$.

The reader should be familiar with the fact that the \emph{band move} $L \to L*b$ yields a cobordism from $L$ to $L*b$ in $S^3 \times [0,1]$.  A sequence of $n$ such cobordisms from a knot $K$ to the  unlink of $n + 1$ components   yields a ribbon disk  in  $B^4$  formed as the union of the cobordism and disjoint disks bounded by the unlink.  Two such sequences yield an embedded sphere formed as the union  of the ribbon disks   in $S^4 = B^4 \cup B^4$.  If the sequences arise from single bands $b$ and $c$, we denote the knotted 2--sphere $(K, b, c)$. We have the following reinterpretations of two special cases of Donald's double slicing criterion~\cite{donald:embedding}.
 
\begin{theorem}\label{thm:don1}
If $K$ is a knot and  $b$ and $c$ are disjoint bands for which $K*b $ is an unlink, $K*c$ is an unlink, and $K*b*c$ is an unknot, then $ (K,  b , c )$ is unknotted.
\end{theorem}

\begin{proof}
Write $U_2 = K*b$ and $U_2' = K*c$.  Both are unlinks.  Write  $U_1= K*b*c$, which is an unknot.     The surface $(K,b,c)$ corresponds to the sequence $$U_2 \to U_2*b = K \to K*c = U_2'.$$  Changing the order of the bands, this can be rewritten as $$U_2 \to U_2 *c \to U_2 *c *b.$$  Since $U_2  = K*b$, we can express this as 
$$U_2  \to K*b *c \to K*b *c *b.$$  Using the facts that $K*b*c = U_1$  and $K*b*c*b = K*b*b*c = K*c$, we finally rewrite the sequence as $U_2 \to U_1 \to U_2'$.  
 
According to Scharlemann~\cite{scharlemann}, a ribbon disk for the unknot with two minima is trivial.  Thus, $(S,b,c)$ is the union of two trivial disks; hence it is the unknot. 
\end{proof}
 
By iterating this approach, one can easily prove results such as the following.
 
\begin{theorem}\label{thm:don2}
Let  $K$ be a knot with disjoint bands $a$, $b$, $c$, and $d$ and suppose that $K*a*b$ and $K*c*d$ are three component unlinks. In this case, there this an associated knotted sphere, $K(a,b,c,d)$.  If $K*a*b*c$ and $K*a*c*d$ are unlinks of two components and $K*a*b*c*d$ is an unknot, then $K(a,b,c,d)$ is unknotted.
\end{theorem}
  
Note that Scharlemann's theorem is used above to show that certain slicing disks for the unknot are trivial.  In each of   the examples we consider, one can quickly show that the relevant slice  disks for the unknot are trivial  by observing that they are built using   trivial band  sums of the unlink;  in particular, for our examples, one need not   use the depth of Scharlemann's theorem.

\subsection{Superslice knots}\label{subsec:superslice}\ 

A knot $K$ is called \emph{superslice} if there is a slice disk $D$ for $K$ such that the double of $D$ along $K$ is an unknotted 2--sphere in $S^4$.

Suppose that $K$ is obtained by attaching a band $\ups$ to an unlink of two components.  See Figure~\ref{fig:SuperRibbon12} for the pertinent three examples.  Let $D_1$ and $D_2$ denote the standard pair of disks bounded by the two-component unlink. In this case, the union  $D=D_1\cup\ups\cup D_2$ is an obvious ribbon disk for $K$.  This disk is immersed in $S^3$ with ribbon singularities, but if we push the interiors of $D_1$ and $D_2$ into $B^4$, we obtain an embedded disk, still called $D$,  with two minima and one saddle  with respect to the standard radial Morse function.  We can assume that $D$ is properly embedded by pushing the entire interior into $B^4$, but pushing the interiors of $D_1$ and $D_2$ in farther.

Let $\K$ be the 2--knot obtained by doubling the disk $D$.  That is, glue two copies of $(B^4,D)$ together along their common $(S^3, K)$ boundary (via the identity map) to get $(S^4, \K)$.  By construction, we see that $\K$ is formed by taking two unknotted 2--spheres $S_1$ and $S_2$ in $S^4$ and attaching a tube $\Upsilon$ that connects them.  Here, $S_i$ is the double of $D_i$ and $\Upsilon$ is the double of $\ups$.

\begin{figure}[h!]
\centering
\includegraphics[scale = .4]{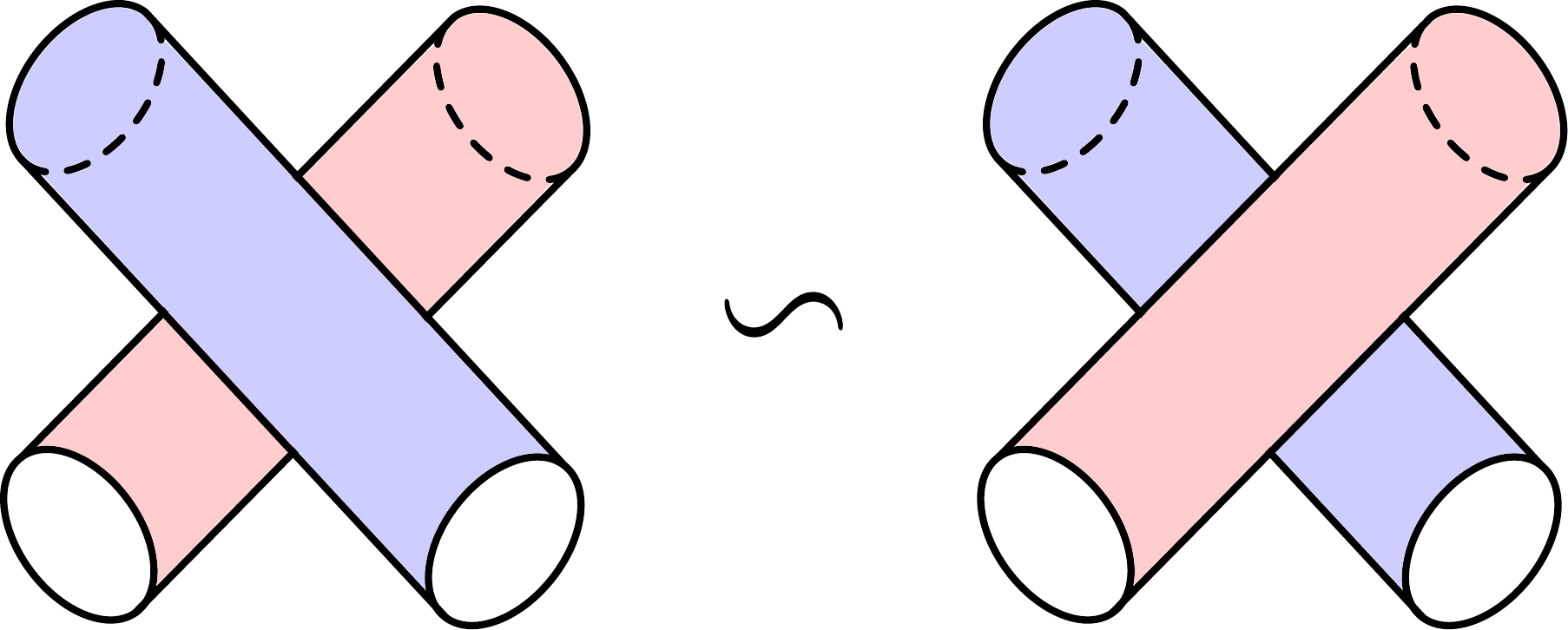}
\caption{A local picture of a 2--knot isotopy that passes one tube through another.}
\label{fig:TubePass}
\end{figure}

Suppose that locally, we see two pieces of $\Upsilon$ as in Figure~\ref{fig:TubePass}; there is an isotopy that passes these two pieces through each other, as shown.  This isotopy corresponds to passing pieces of $\upsilon$ past each other.  This changes the isotopy class of the band $\ups$, giving a new band $\ups'$  and a new ribbon knot $K'$, which is obtained by attaching $\ups'$ to the original unlink.  Because this change resulted from an isotopy of $\K$, we see that both $K$ and $K'$ are cross-sections of $\K$.  If $K'$ is unknotted, then $\K$ is unknotted, as in the proof of Theorem~\ref{thm:don1}, and we can conclude that $K$ is doubly slice.  We summarize this with the following criterion.

\begin{proposition}\label{prop:super}
	Let $K$ be a knot that is obtained by attaching a single band $\ups$ to an unlink of two components.  Let $K'$ be the result of passing the band $\ups$ through itself as discussed above.  If $K'$ is the unknot, then $K$ is smoothly superslice.  In particular, if the band is relatively homotopic to a trivial band in the complement of a neighborhood of the unlink, then $K$ is smoothly superslice.
\end{proposition}

Figure~\ref{fig:SuperRibbon12} shows three examples of ribbon knots that satisfy the above criterion and can therefore be seen to be smoothly superslice.

\begin{corollary}
	The isotopy class of $\K$ depends only on the homotopy class of the core of $\ups$.
\end{corollary}

\begin{figure}[h!]
\centering
\includegraphics[scale = .5]{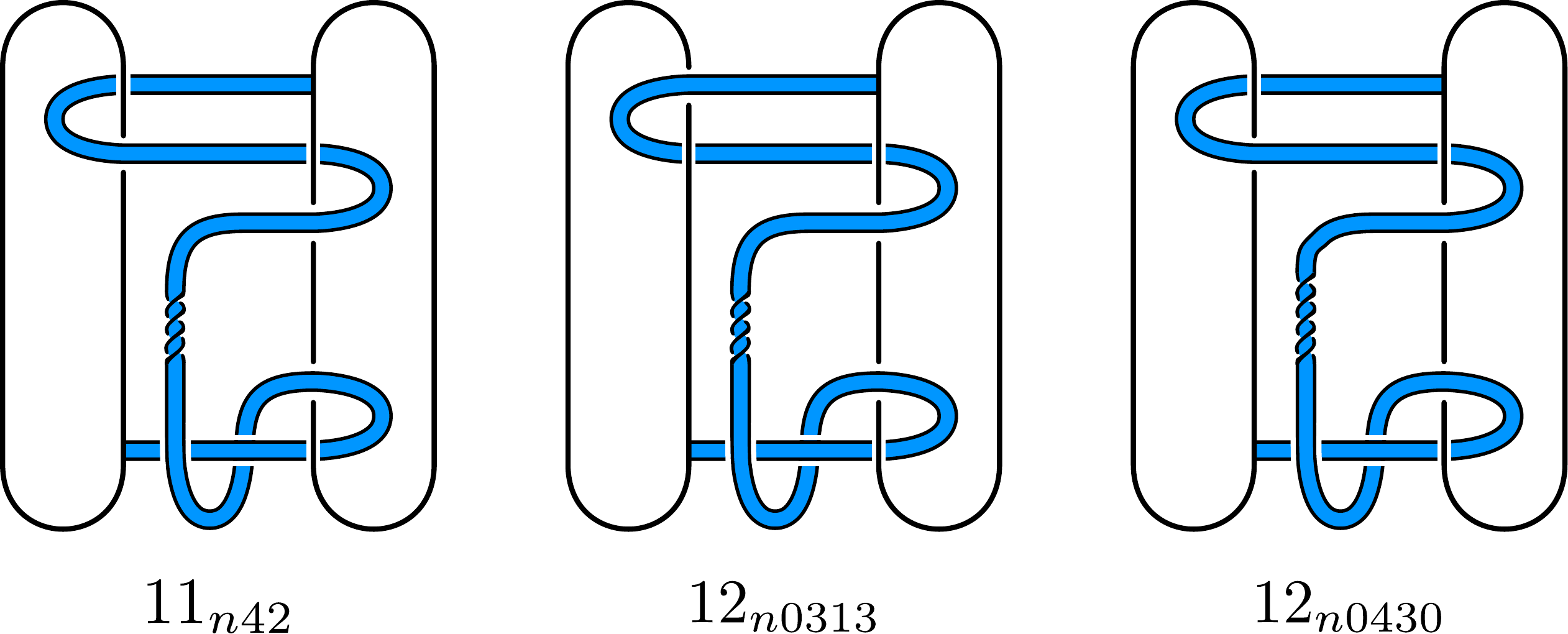}
\caption{The above knots are smoothly superslice.  See Subsection \ref{subsec:superslice}.}
\label{fig:SuperRibbon12}
\end{figure}

\subsection{Freedman and the locally flat setting}\label{subsec:freedman}\ 

Let $K$ be a knot in $S^3$, and let $\Delta_K(t) $ denote the Alexander polynomial of $K$.  It is a well-known consequence of the work of Freedman and Quinn that any knot $K$ with $\Delta_K(t) =1$ bounds a topologically locally 
flat disk in $B^4$~\cite{freedman:concordance, freedman-quinn}.  
In fact, a stronger, yet less well-known, fact is true. (See~\cite{meier:double} for more detail.)  

\begin{theorem}\label{thm:Freedman}
Let $K$ be a knot in $S^3$.  If $\Delta_K=1$, then $K$ is topologically superslice.
\end{theorem}

There are five knots, up to 12 crossing, with trivial Alexander polynomial. The first is the Conway knot $11_{n34}$.  
Theorem~\ref{thm:Freedman} shows that this knot is topologically doubly slice. Interestingly, it turns out that each of the other four knots is smoothly superslice; the double of the ribbon disk is an unknotted 2--sphere in $S^4$.  Thus we are led to the following problem, which at the moment seems inaccessible. 
\begin{problem} Find a smoothly slice knot $K$ with $\Delta_K(t) =1$ that is not smoothly superslice.
\end{problem}
Superslice knots were first studied by Gordon-Sumners~\cite{gordon-sumners}, who showed that the double of any slice knot is superslice and that for any superslice knot $K$, $\Delta_K(t) = 1$. 
Superslice knots were also were studied in relation to the Property R Conjecture~\cite{brakes,gordon:satellite, kirby-melvin:R}.

We remark that many infinite families of superslice knots can be created by taking any properly embedded arc in the complement of an unlink that is homotopic, but not isotopic, to the trivial arc connecting the two components and banding along the arc with some framing.  Changing the framing produces infinitely many knots in each family which can be distinguished from each other by their Jones polynomials.  For example, any of the three knots shown in Figure~\ref{fig:SuperRibbon12} gives rise to  such a family by adding twists to the band in each case.

\subsection{Proof of Theorem~\ref{thm:Smooth}}\ 

We are now equipped to prove our main result.

\begin{reptheorem}{thm:Smooth}
The following knots are smoothly doubly slice.
$$
\begin{array}{lllllll}
	9_{46} & 10_{99} & 10_{123} & 10_{155} & 11_{n42} & 11_{n49} & 11_{n74} \\
	12_{a0427} & 12_{a1105}  & 12_{n0268} & 12_{n0309} & 12_{n0313} & 12_{n0397} & 12_{n0414} \\
	 12_{n0430} & 12_{n0605} & 12_{n0636} & 12_{n0706} & 12_{n0817} & 12_{n0838} 
\end{array}
$$
Furthermore, the following are the only other prime knots with 12 crossings or fewer that could possibly be smoothly doubly slice.
$$
\begin{array}{lllllll}
	11_{n34} & 11_{n73} & 12_{a1019} & 12_{a1202} 
\end{array}
$$
\end{reptheorem}

\begin{proof}
The Kinoshita-Terasaka knot $11_{n42}$  was shown to be smoothly superslice in~\cite{carter-kamada-saito}.  
Figure~\ref{fig:SuperRibbon12} shows ribbon disks for $12_{n0313}$ and $12_{n0430}$.  It is easy to see that that each knot satisfies the hypotheses of Proposition~\ref{prop:super}; therefore, each of these knots is smoothly superslice, hence smoothly doubly slice.

The remaining 17 knots are shown in Figures~\ref{fig:TwoBand1},~\ref{fig:TwoBand2}, and~\ref{fig:TwoBand3}. With the exception of $12_{n0636}$, these knots all satisfy Theorem~\ref{thm:don1}.  The knot $12_{n0636}$ requires a pair of two-band systems, and is smoothly doubly slice by Theorem~\ref{thm:don2}.  (Note that the order in which the two-band systems are resolved doesn't matter in this case.)
\end{proof}

Portions of Theorem~\ref{thm:Smooth}  were previously
 known:  $9_{46}$ was first smoothly doubly sliced by Sumners~\cite{sumners:invertible}, 
 $11_{n42}$ by Carter, Kamada, and Saito~\cite{carter-kamada-saito}, and $10_{123}$ (private communication) and $11_{n74}$ 
 (in~\cite{donald:embedding}) by Donald.   



\begin{question}\label{question:Smooth}
Are any of the following knots smoothly doubly slice?
$$
\begin{array}{lllllll}
	11_{n34} & 11_{n73} & 12_{a1019} & 12_{a1202}
\end{array}
$$
\vskip.05in
\end{question}

Recall that $11_{n34}$ is  topologically doubly slice.  Other than this, Question~\ref{question:Smooth} applies equally well in the topological setting and covers all possibilities.  This completes our analysis. 

\section{The double slice genus of knots}\label{sec:genus}

The study of doubly slice knots  can be placed in  the broader context of the relationship between knots in the 3--sphere and surfaces in the 4--sphere.  In this section, we will briefly describe this more general setting.

Let $\mathcal S$ be an orientable surface in $S^4$.  We say that $\mathcal S$ is \emph{unknotted} if $\mathcal S$ bounds a handlebody $H$ in $S^4$.  Let $\mathcal S$ be an unknotted surface in $S^4$, and suppose that $\mathcal S$ transversely intersects the standard $S^3$ in a knot $K$.  We say that $K$   \emph{divides}  $\mathcal S$.

Let $K$ be a knot in $S^3$ and let $F$ be a Seifert surface for $K$ with $g(F)=g$.  We think of $F\subset S^3\subset S^4$, where $S^3$ lies as the equator of $S^4$. Let $H=F\times[-1,1]$, with $H\cap S^3=F$; the surface $F$ is the intersection of a handlebody $H\subset S^4$ with $S^3$.  Let $\mathcal S=\pd H$. Then, $\mathcal S$ is an unknotted surface in $S^4$ (by definition) and $K=\mathcal S\cap S^3$.  It follows that every knot $K$ in $S^3$ divides an unknotted surface in $S^4$.

Therefore, we define $$g_{ds}(K)=\min\{g(\mathcal S)\ |\ \mathcal S\subset S^4, \text{\ $\mathcal S$ unknotted, and } \mathcal S\cap S^3=K \}.$$
We call $g_{ds}(K)$ the \emph{double slice genus} of $K$.  Note that $g_{ds}(K)=0$ if and only if $K$ is doubly slice.  Furthermore, we saw above that $g_{ds}(K)\leq 2g_3(K)$.  Similarly, it is clear that $2g_4(K)\leq g_{ds}(K)$.

The restriction $2g_4(K)\leq g_{ds}(K)\leq 2g_3(K)$ is already enough to determine the double slice genus for a third of the knots up to nine crossings.  A more detailed analysis will be the subject of future study by the authors.

\begin{figure}
\centering
\includegraphics[scale = .5]{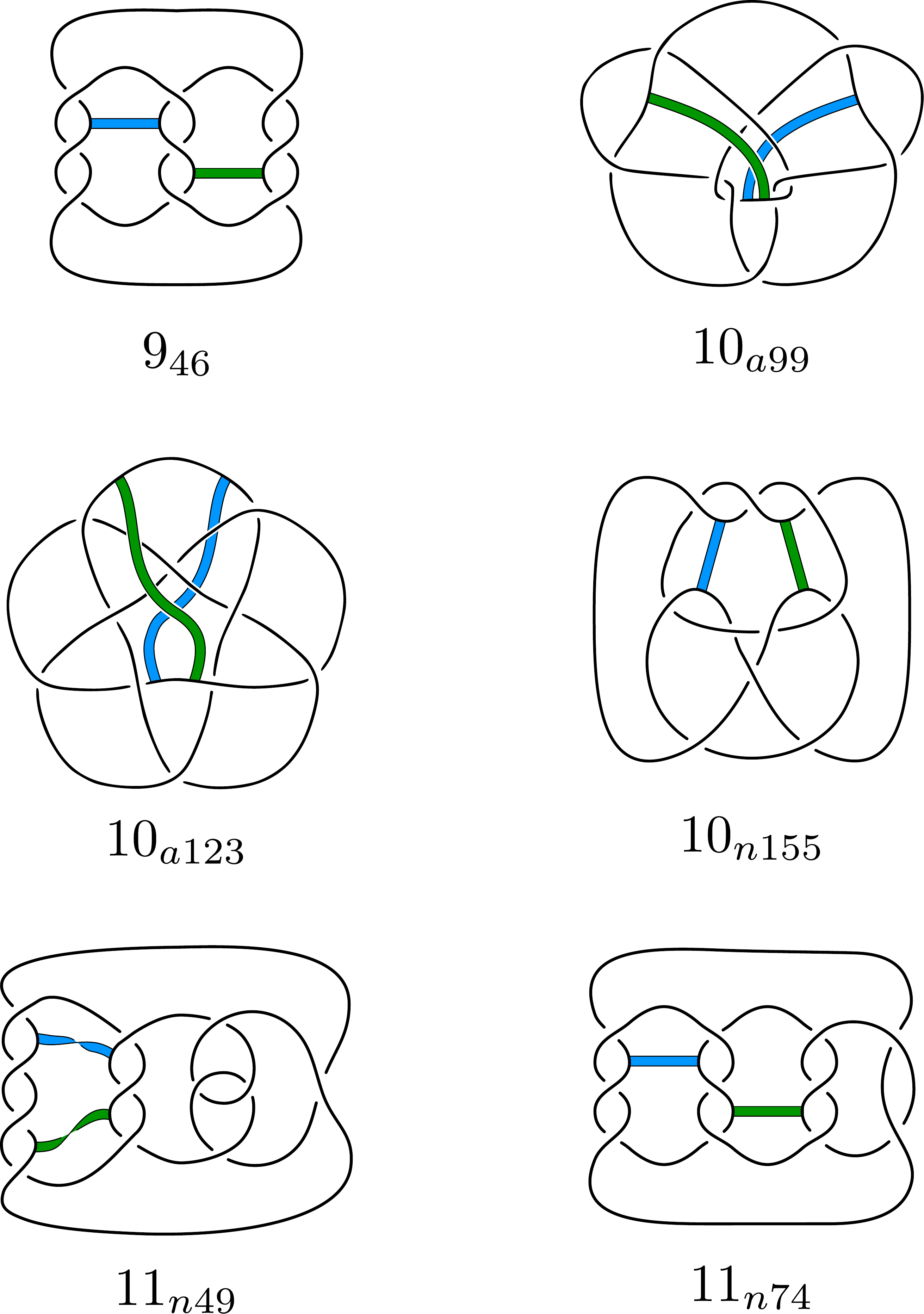}
\caption{The above knots are smoothly doubly slice.  See Subsection~\ref{subsec:bands}.}
\label{fig:TwoBand1}
\end{figure}

\begin{figure}
\centering
\includegraphics[scale = .43]{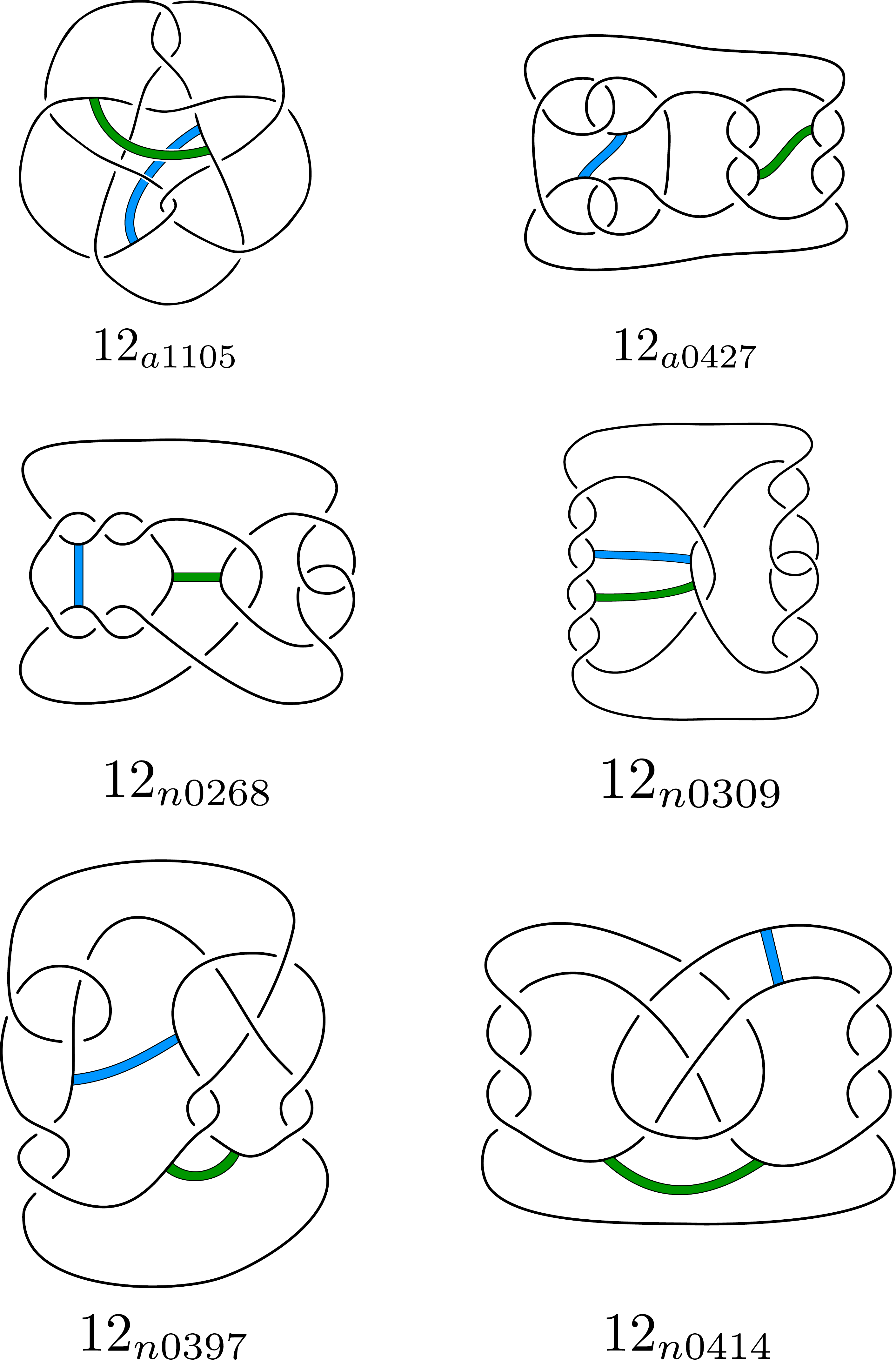}
\caption{The above knots are smoothly doubly slice.  See Subsection~\ref{subsec:bands}.}
\label{fig:TwoBand2}
\end{figure}

\begin{figure}
\centering
\includegraphics[scale = .4]{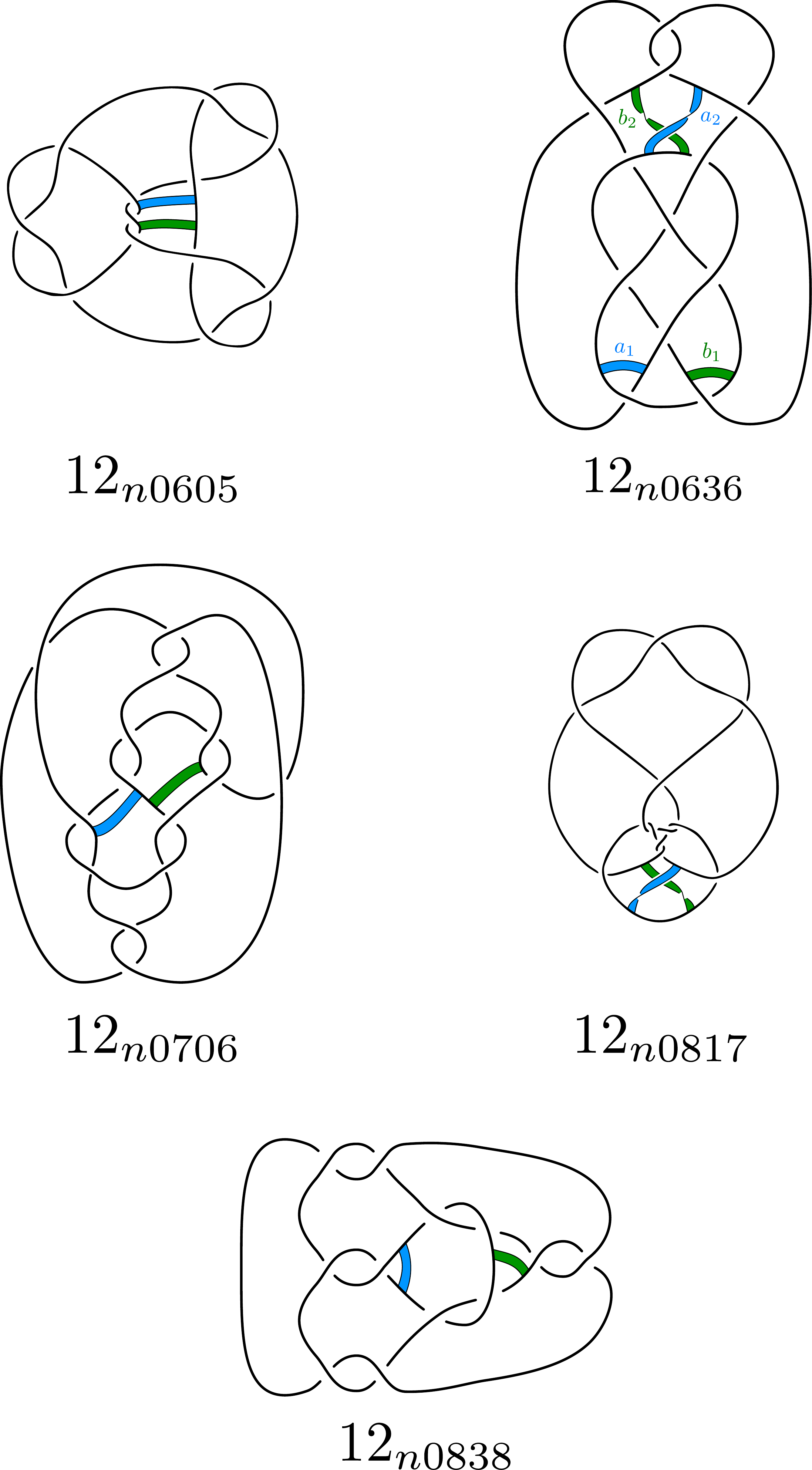}
\caption{The above knots are smoothly doubly slice.  See Subsection~\ref{subsec:bands}.}
\label{fig:TwoBand3}
\end{figure}

\clearpage

\bibliographystyle{abbrv}

\bibliography{UpTo12Biblio}

\end{document}